\documentclass[a4paper,12pt, legno]{article}
\usepackage{amsmath}
\usepackage{amsfonts}
\usepackage{amsthm}
\usepackage{mathrsfs}
\usepackage[T1]{fontenc}
\usepackage{amsmath}
\usepackage{amsfonts}
\usepackage{amsthm}
\usepackage{mathrsfs}
\usepackage{amssymb}
\usepackage{pst-node}
\usepackage{tikz-cd}

\newtheorem{thm}{Theorem}
\newtheorem{prop}[]{Proposition}
\newtheorem{lem}[]{Lemma}

\newtheorem{cor}{Corollary}

\newtheorem{ex}[]{Eample}
\newtheorem{rem}[]{Remark}

\newcommand{\Rset}{\mathbb{R}}
\newcommand{\Cset}{\mathbb{C}}

\begin{document}

\title{Urbanik type subclasses of  free - infinitely  divisible transforms. }
\author{Zbigniew J.  Jurek  (University of Wroc\l aw)}
\date{August 23, 2021.}
\maketitle
\medskip
[ to appear in \emph{Probab. Math. Statistics}, vol. 42,(2022)]

\medskip
\medskip
\emph{\underline{ABSTRACT}.}

For the class of free-infinitely divisible transforms  are introduced  three  families  of increasing  Urbanik type  subclasses of those transforms. They begin with the class of free-normal transforms and end up  with the whole class of  free-infinitely divisible transforms. Those subclasses are derived from the ones of classical infinitely divisible measures for which are known their random integral representations. Special functions like Hurwitz-Lerch,  polygamma and hypergeometric  appear in kernels of the corresponding integral representations.



s-selfdecomposability; L\'evy-Khintchine formula; free-infinite divisibility;

Nevanlinna-Pick functions;  polygamma function;

hypergeometric function.

\medskip
\underline{Running title}: Urbanik type subclasses ...

\medskip
\medskip
\underline{Author's addresses}:

Institute of Mathematics

University of Wroc\l aw

Pl. Grunwaldzki 2/4

50-384 Wroc\l aw,  Poland.

www.math.uni.wroc.pl/$\sim$zjjurek ;  zjjurek@math.uni.wroc.pl

\newpage
The limit distribution theory is one of the main topics in the probability theory. Historically, it began  with the central limit theorem which says that properly normalized partial sums of\emph{ independent and identically distributed} (i.i.d)  random variables with  finite second moment, converge in distribution to \emph{the standard normal (Gaussian)} variable. When one drops the assumption about  moments but still assumes i.i.d. variables, at the limit we get the class of \emph{stable distributions} (variables). Still further, if we assume that observations (variables)  are only stochastically independent but after some normalization by positive constants the corresponding triangular is \emph{uniformly infinitesimal} then at the limit we obtain the class of \emph{selfdecomposable distributions} ( L\'evy class $L$). Finally,  limits of sums of arbitrary infinitesimal and row-wise independent triangular arrays coincide with the class of \emph{infinitely divisible distributions}; see  Feller (1966), Chapter XVII or  Gnedenko and Kolmogorov (1954), Sect. 17-19, \ 29-30 and  33 or Loeve (1963), Sect. 23.
Thus we have
\[(\emph{Gaussian})\subset ...\subset(\emph{selfdecomposable})\subset...\subset(\emph{infinitely divisible}) \,\, (\star)
\]
[Here it might be worthy to notice that the class of selfdecomposable measures can be obtained from \emph{ strongly mixing sequences}  not necessarily stochastically independent; cf. Bradley and Jurek (2014). That possible direction of studies is not continued in this article.]

\medskip
Urbanik (1972, 1973) refined the left hand side inclusion (below class $L$)  and Jurek (1988, 1989) the right hand side inclusion by introducing new  subclasses of some limit distributions.

Later on, all the introduced subclasses were generalized to distributions on infinite dimensional spaces and then they were
described as distributions of  some random integrals on arbitrary Banach spaces. In Jurek (1983a), the normalization of partial sums of random variables was done by  linear bounded operators on a Banach space. Those and other multidimensional set-ups might be of some use in a generalization to multidimensional free-probability theory.
[For  the random integral representation conjecture see the link given below the reference Jurek (1985.]

In this paper we give characterizations of the above mentioned results (i.e., the refinements of inclusions in $(\star)$) for the  additive free-independence (free-additive convolution $\boxplus$). More precisely, we describe  its corresponding  free-independent (Voiculescu) transforms. Those transforms are considered only on the imaginary axis which is enough for the identification of the corresponding measure; see Jurek (2006), Jankowski and Jurek (2012).

In Theorem 1 and auxiliary  Lemma 1 we  treat the general random integrals mappings that lead to  subsets of free-independent transforms . Propositions 1-4 provide  applications of Theorem 1 to some specified mappings.  A complete filtration of the class of all free-infinitely divisible transforms is given  in  Corollary 2. Finally, Theorem 2 shows an intrinsic relation between two classes of free-infinitely divisible transforms : one derived from \emph{linear} scalings and the other from  \emph{non-linear} scalings.

\medskip
\medskip
\textbf{0. Introduction and notations.}

We will introduce Urbanik type subclasses of  free-infinitely  divisible

\noindent Voiculescu  transforms in a such way that
\begin{multline}
(Gaussian, \boxplus)\subset (stable, \boxplus)\subset ... \subset(\mathcal{U}^{<k+1>}, \boxplus)\subset (\mathcal{U}^{<k>},\boxplus) \\ \subset ... \subset (\mathcal{U}^{<2>}, \boxplus)\subset (\mathcal{U}^{<1>},\boxplus)\equiv (\mathcal{U},\boxplus)\\ \equiv   (\mathbb{U}_1,\boxplus) \subset ... \subset (\mathbb{U}_k,
\boxplus)\subset(\mathbb{U}_{k+1},\boxplus)\subset ... \subset\overline{\cup_{k=1}^\infty (\mathbb{U}_k,\boxplus)}\equiv (ID,\boxplus),
\end{multline}
where the closure is in the point-wise convergence of Voiculescu transforms (the topology of weak convergence of measures) and $\boxplus$ is the free-additive convolution.

Classes $(\mathcal{U}^{<k>},\boxplus)$ and  $(\mathbb{U}_k, \boxplus)$ are the free-probability counterparts of the classical probability classes $ (\mathcal{U}^{<k>},\ast)$ and $ (\mathbb{U}_k, \ast) $ in $(ID,\ast)$. For each of these classes we have their characterizations in terms of the random integrals; for a general conjecture see the link below reference Jurek (1985). 

For the above classes  we have  the following integral representations
\begin{multline}
(i) \ \ \mu\in (\mathcal{U}^{<k>},\ast) \ \mbox{iff} \ \mu=\mathcal{L}\Big(\int_{(0,1]}tdY(\tau_k(t))\Big), \\ \tau_k(t):= \frac{1}{(k-1)!}\int_0^t(- \log v)^{k-1} dv; \  0<t  \le 1 ; \  \
 k=1,2,... \\ (ii)\ \  \nu\in (\mathbb{U}_k,\ast) \ \mbox{iff} \ \mu=\mathcal{L}\Big(\int_{(0,1]}tdY(r_k(t))\Big), \ r_k(t):=t^k, \ 0\le t\le 1,
\end{multline}
and $(Y(t), t\ge 0)$ is a cadlag L\'evy process and $\mathcal{L}(Z)$ denotes a probability distribution ( a law) of a random variable $Z$.

\noindent [Although $\int (- \log x)^k dx=\Gamma(k,-\log x)+ const $, (the incomplete Euler gamma function), we do not use  that identity here.]

The above, (i) and (ii), are  particular examples of  random integrals
\begin{equation}
\rho=I^{h,r}_{(a,b]}(\mu):=\mathcal{L}\Big(\int_{(a,b]}h(t)dY(r(t))\Big), \ \  \mathcal{L}(Y(1))=\mu\in \mathcal{D}^{h,r}_{(a,b]}
\end{equation}
where $h$ is a real function, $r$ (a time change)  is a monotone, nonnegative function and $\mathcal{D}^{h,r}_{(a,b]}$ denotes  the domain of  a random integral $I^{h,r}_{(a,b]}$; for details see, for instance, Jurek (1988) or (1989) or (2004) or (2007) or (2018). To $Y$ (to $\mu$) we refer as  the \emph{background driving L\'evy process} (BDLP) (\emph{ the background driving probability distribution}(BDPD)) of the measure $\rho$.

The identification (the isomorphism) between classical infinitely divisible characteristic functions $\phi_{\mu}(t)$ and their counter part Voiculescu free-infinitely divisible transforms $V_{\tilde{\mu}}(it)$ (or measures )  is given as follows:
\begin{equation}
(ID,\ast)\ni \mu \to V_{\tilde{\mu}}(it)=it^2\int_0^\infty\log\phi_\mu(-u)e^{-tu}du,  \ \ t>0;
\end{equation}
see Jurek (2007), Corollary 6 and the random integral mapping $\mathcal{K}^{(e)}$ which was the origin for the identity (4). The need for a such  identification arises when one wants to use Bercovici-Pata isomorphism but we do not have parameters $a$ and $m$, in the L\'evy-Khintichine or Bervovici-Voiculescu formula, for the respective classical and free independence. That those two approaches do coincide was shown in Jurek (2007), Jurek (2016) or Jurek (2020), Theorem 2.1.  Moreover, in  Jurek (2020), on page 350, is given a diagram how one may connect any two abstract semigroups.

\medskip
From the above  mapping (4)  we infer the  properties
\[
V_{\widetilde{\mu\ast\nu}}(it)= V_{\tilde{\mu}}(it)+ V_{\tilde{\nu}}(it)= V_{\tilde{\mu}\boxplus \tilde{\nu}}(it); \ \ \mbox{for $c>0$,} \ \ V_{\widetilde{T_c\mu}}(it)=cV_{\tilde{\mu}}(it/c),
\]
and the last property is in the sharp contrast with $\phi_{T_c\mu}(t)=\phi_\mu(ct)$, for the characteristic functions.

The fundamental L\'evy-Khintchine characterization says that
\begin{multline}
\mu \in ID \ \ \mbox{iff}\ \   \phi_\mu(t)=\exp\big[ ita+\int_{\Rset}(e^{itx}-1-\frac{itx}{1+x^2})\frac{1+x^2}{x^2}m(dx)\big] \\
= \exp\big[ ita- \frac{1}{2}\sigma^2 t^2+ \int_{\Rset\setminus\{(0)\}}(e^{itx}-1-\frac{itx}{1+x^2})M(dx) \big], \ \ t\in\Rset,
\end{multline}
where parameters: a  real number $a$, a finite Borel measure $m$  and a L\'evy spectral measure $M$ are uniquely determined. In the latter case, we will simply  write $\mu=[a,\sigma^2,M]$.

For  the free-infinite divisibility here  is the Voiculescu  analogue of the L\'evy-Khintchine formula:
\begin{equation}
\nu\in(ID,\boxplus)\ \mbox{iff} \  V_{\nu}(it)=a +\int_{\Rset}\frac{1+ itx}{it-x}m(dx), \, t \neq 0,
\end{equation}
where the parameters: a real number $a$ and a finite Borel measure $m$ are uniquely determined; for details see Voiculescu (1986), Bercovici and Voicu-lescu (1993),  Barndorff-Nielsen and  S. Thorbjorsen (2006)  and Jurek (2006), (2007).

Uniqueness of  parameters $a$ and $m$ in the formulas (5) and (6) give an natural identification between classical and free-infinitely divisible transforms as it was already mentioned above. On the other hand, if one knows that $\phi_{\mu} \in (ID,\ast)$ or $V_{\nu}(it) \in (ID,\boxplus)$ finding their corresponding parameters $a$ and $m$ might be, in general, quite difficult. In that case, for $(ID,\ast)$ and $(ID,\boxplus)$ we use the identification given by (4).

\medskip
\medskip
\textbf{1. A  basic theorem.}

Here is a basic result which will allow us  to introduce effectively   new subclasses of free-infinitely divisible transforms  by specifying their corresponding random integral representations $I^{h,r}_{(a,b]}(\mu)$; see  (3), above.
\begin{thm}
For deterministic functions $h$ and $r$ on an interval $(a,b]$, let us define the constants \textbf{c}, \textbf{d} and the function \textbf{g}$_{\pm}$ depending on them (i.e., on  $h$ and $r$) as follows:
$$\textbf{c}:=\int_a^b  h(s)dr(s), \ \textbf{d}:=\int_a^b  h^2(s)dr(s), \ \  \textbf{g}_{\pm}(z):=\int_a^b\frac{h(s)}{z\,h(s)\pm 1}dr(s),  \ z  \in \Cset.$$
Then for $\mu=[a,\sigma^2,M] \in\mathcal{D}^{h,r}_{(a,b]}$  and $\rho:=I^{h,r}_{(a,b]}(\mu)$ there exists a counter part measure $\tilde{\rho}\in (ID,\boxplus)$ such that
\begin{equation}
V_{\tilde{\rho}}(it)= a\, \textbf{c} + (\pm) \frac{\sigma^2}{it}\,\textbf{d}+\int_{\Rset\setminus{(0)}}(\pm) x[\textbf{g}_{\pm}(\frac{ix}{t}) -  \frac{(\pm)\textbf{c}}{1+x^2}]M(dx), \ \ t>0,
\end{equation}
where, in the pair $ (\pm)$, the upper sign is for non decreasing $r$ and the lower sign for non increasing $r$, respectively.

Equivalently,by putting $ m(dx):=\frac{x^2}{1+x^2}M(dx)$ on $\Rset\setminus\{0\}$  and $m(\{0\}):=\sigma^2$ we get a finite measure $m$ such that
\begin{equation}
V_{\tilde{\rho}}(it)= a\, \textbf{c}+\int_{\Rset} (\pm)\big[\textbf{g}_{\pm}(\frac{ix}{t})- \frac{(\pm) \textbf{c}}{1+x^2}\big]\frac{1+x^2}{x}m(dx), \ \ t>0,
\end{equation}
where the integrand in (8) at zero is equal $ (\pm) \textbf{d}(it)^{-1}$.

Moreover, if $h(s)>0$, we have  the following relation between $\tilde{\rho}$ and $\tilde{\mu}$, the free-probability counterparts of $\rho$ and its background driving measure $\mu$, respectively.
\begin{equation}
V_{\tilde{\rho}}(it)=\int_a^bh(s)V_{\tilde{\mu}}(it/h(s))dr(s)=\int_a^bV_{\widetilde{T_{h(s)}\mu}}(it)dr(s), \  t>0,
\end{equation}
where $(T_c(\mu))(B):=\mu(c^{-1}B), c>0,$ for Borel sets $B$.

\end{thm}

\begin{proof}
The isomorphism  (4) gives one-to-one correspondence  between the classical $\phi_\mu(t)$ and the free-infinity divisible $V_{\tilde{\mu}}(it)$  transforms.  The law ( of the random  integral)\  $\rho=I^{h,r}_{(a,b]}(\mu)$ has the  characteristic function
\begin{equation}
  \log\phi_{\rho}(v)=\int_a^b\log\phi_{\mu}((\pm) h(s)v) (\pm) dr(s),
\end{equation}
where $(\pm)$ is for a non-decreasing and a non- increasing  function $r$, respectively.

\medskip
Since by (5),
 $\log\phi_{\mu}(t)=it a-\frac{\sigma^2}{2}t^2+ \int_{\Rset\setminus\{(0)\}}\big(e^{itx}-1- \frac{itx}{1+x^2}\big)M(dx),$
therefore, using the Fubini Theorem and (4),  we get

\begin{multline*}
V_{\tilde{\rho}}(it)=  it^2\int_0^\infty \log\phi_{\rho}(- u)e^{-tu}du\\ =   it^2\int_0^\infty \int_a^b \log\phi_{\mu}(- (\pm) h(s)u)(\pm)dr(s)\, \,e^{-tu}du\\ = \int_a^b it^2\int_0^\infty [  \log\phi_{\mu}(\mp h(s)u)\,e^{-tu}du]\,(\pm)dr(s) \\ = \int_a^b  i a(\mp h(s)) it^2\int_0^\infty u e^{-tu}du(\pm) dr(s)  \\ -
\int_a^b (\mp h(s))^2\frac{1}{2}\sigma^2it^2\int_0^\infty u^2e^{-tu}du(\pm)dr(s) \\ +
\int_a^b \int_{\Rset\setminus\{(0)\}} it^2\int_0^\infty\Big( e^{(\mp) ih(s)xu}-1 - \frac{x}{1+x^2}((\mp) ih(s) u)\Big) e^{-tu}duM(dx)(\pm) dr(s)\\=
\int_a^b  i a((\mp)h(s)) it^2 \frac{1}{t^2}(\pm) dr(s) -
\int_a^b (h(s))^2\frac{1}{2}\sigma^2it^2 \frac{2}{t^3}(\pm) dr(s)\\+
\int_a^b \int_{\Rset\setminus\{(0)\}} it^2 \Big( \frac{1}{t\pm ih(s)x}-\frac{1}{t} - \frac{i(\mp)h(s)x}{1+x^2}\frac{1}{t^2}\Big)M(dx)(\pm) dr(s)\\
= a\,\textbf{c} +\frac{\sigma^2}{it} (\pm) \textbf{d} +
\int_{\Rset\setminus\{(0)\}}\int_a^b[ \frac{(\pm)t\,h(s)x}{t\pm i xh(s)}+ \frac{x}{1+x^2}(\mp)h(s)](\pm) dr(s)M(dx)\\=
a\,\textbf{c} +\frac{\sigma^2}{it} (\pm) \textbf{d} + \int_{\Rset\setminus\{(0)\}}  (\pm)x\int_a^b [ \frac{h(s)}{1\pm i xh(s)/t}+ \frac{1}{1+x^2}(\mp)h(s)]dr(s)M(dx)\\=
a\textbf{c}+\frac{\sigma^2}{it} (\pm)\textbf{d}+ \int_{\Rset}(\pm)x\,[\textbf{g}_{\pm}(ix/t)- \frac{ (\pm)\textbf{c}}{1+x^2}]\,M(dx),
\end{multline*}
which proves the formula (7).

To get (8), let us  note that $g_{\pm}(0)= \pm \textbf{c}$  and
 $$\lim_{x\to 0} (\pm)[\frac{\textbf{g}_{\pm}(\frac{ix}{t})-\frac{(\pm)\textbf{c}}{1+x^2}}{x}]= (\pm)\lim_{x\to 0}\int_a^b\frac{(h(s))^2}{( ixh(s)/t \pm 1)^2}\,\,\frac{-i}{t}dr(s)= \frac{1}{it} (\pm) \textbf{d}.$$

Finally, similarly as above,  we have
\begin{multline*}
V_{\tilde{\rho}}(it)= \int_a^b it^2\int_0^\infty [  \log\phi_{\mu}(- h(s)u)\,e^{-tu}du]\,dr(s)\\= \int_a^b it^2\int_0^\infty [  \log\phi_{\mu}(- w)\,e^{-tw/h(s)}\frac{dw}{h(s)}]\,dr(s)=\int_a^b h(s)V_{\tilde{\mu}}(it/h(s) dr(s)\\=\int_a^b V_{\widetilde{T_{h(s)}\mu}}(it)dr(s), \ \mbox{as by (4),} \ \ V_{\widetilde{{T_c\mu}}}(it)= c V_{\tilde{\mu}}(it/c), \ c>0;
\end{multline*}
which completes a proof.

\end{proof}
 The kernel $\textbf{g}(z)$ from Theorem 1 admits the following representation:

\begin{lem}
For non-decreasing  $r$, functions $\textbf{g}(z):=\int_a^b\frac{h(s)}{zh(s)+1}dr(s)$ map upper half-plane of $\Cset^+$ into lower half-plane $\Cset^{-}$  are analytic ones  with  the  Pick - Nevanlinna  representation
$$
\textbf{g}(z)=\int_a^b\frac{h(s)}{1+ h^2(s)}dr(s)+ \int_{\Rset}\frac{1+zx}{z-x}(\int_a^b\frac{h^2(s)}{1+h^2(s)}\delta_{-1/h(s)}(dx)dr(s))
$$
\end{lem}

\begin{proof}
First, note that $\Im(\textbf{g}(z))=-\Im(z)\int_a^b\frac{h^2(s)}{|1+zh(s)|^2}dr(s)$, where the integral is positive as $r$ is non decreasing.
This means that $\textbf{g}:\Cset^+\to \Cset^{-}$  and is an analytic function with
$$ \frac{d^n}{dz^n}\textbf{g}(z)= (-1)^n n!\int_a^b\big(\frac{h(s)}{1+zh(s)}\big)^{n+1} dr(s), n=0,1,2,..$$

Second,  for $b \in \Rset$ we have  explicit Pick-Nevanlinna representation
$$\frac{1}{z+b}= u_b + \int_{\Rset}\frac{1+zx}{z-x}m_b(dx), \ \   u_b:= \frac{b}{1+b^2}, \ \  m_b(A):=\frac{1}{1+b^2}\delta_{-b}(A),$$
and hence taking $b:=1/h(s)$ and  integrating the above  with respect to $dr(s)$ we get
\begin{equation*}
\textbf{g}(z)=u+\int_{\Rset}\frac{1+zx}{z-x}m(dx), \ \mbox{where} \ \ u:=\int_a^b\frac{h(s)}{1+ h^2(s)}dr(s),
\end{equation*}
is the shift parameter and the measure $m$ is
\begin{equation*}
 m(A):=\int_a^b\frac{h^2(s)}{1+h^2(s)}\delta_A( -\frac{1}{h(s)}) dr(s)\\ =\int_a^b\frac{h^2(s)}{1+h^2(s)}\delta_{-\frac{1}{h(s)}}(A) dr(s),
\end{equation*}
is a mixture of the point-mass Dirac measures, $k(s)\delta_{f(s)}(A)$  which gives the proof of  Lemma 1.
\end{proof}

\medskip
\textbf{2. Classes $(\mathcal{U}^{<k>},\boxplus)$ of free-infinitely divisible transforms (measures) for $ k=1, 2, ...$.}

\medskip
\medskip
For an one-parameter semigroup $(U_r,r>0)$ of  non-linear \emph{shrinking operations} (in short: \emph{s-operations}) defined as follows
$$
U_r:\Rset\to \Rset \ \mbox{ as} \  \ U_r(0):=0, \ \  U_r(x): = \max\{|x|-r,0\}\frac{x}{|x|},  \ \ x \neq 0; r>0,
 $$
in Jurek (1977,1981) was introduced the class $\mathcal{U}$ of limiting distributions  of sequences
\begin{equation*}
U_{r_{n}}(X_1)+U_{r_{n}}(X_1)+...+U_{r_{n}}(X_n)+x_n,
\end{equation*}
where terms $U_{r_n}(X_j), \ 1\leq j\leq n$ , are uniformly infinitesimal and random variables $ X_n, n=1,2,...$  are stochastically independent. Measures  $\mu\in \mathcal{U}$ were termed as \emph{s-selfdecomposable measures.}
\begin{rem}
\emph{Note that nowadays, in the mathematical finance,  for $X>0$,  the s-operation $U_r(X)=(X-r)_+$ is called \emph{the European call option} on a stock $X$ with an exercise price $r$.}
\end{rem}
In Jurek (2004) were introduced and characterized the following  subclasses of the class $(ID,\ast)$ of the classical infinitely divisible measures:
\begin{equation*}
 ... \subset \mathcal{U}^{<k+1>} \subset \mathcal{U}^{<k>}\subset ...\subset \mathcal{U}^{<1>}\equiv \mathcal{U}\subset ID,
\end{equation*}
and the measures $\mu\in \mathcal{U}^{<k>}$  were called \emph{ k-times s-selfdecomposable measures}. Furthermore, as mentioned in the Introduction, taking  the time change
\begin{equation}
\tau_k(t):=\frac{1}{(k-1)!}\int_0^t (-\log v)^{k-1}dv, \ \mbox{we get}   \ ( \mathcal{U}^{<k>},\ast)=I^{t,\tau_k(t)}_{(0,1]}(ID),
\end{equation}
and  $ I^{t,\tau_k(t)}_{(0,1]}(\nu)= I^{t,t}_{(0,1]}(I^{t,t}_{(0,1]}(...(I^{t,t}_{(0,1]}(\nu)))$,  (k-times);    see Jurek (2004), Proposition 4 and Corollary 2 and for more general theory of compositions of random integrals see Jurek (2018).

Here are the free-infinity  divisible counterparts of k-times s-selfdecomposable probability measures:
\begin{prop}
For $ k= 1,...$, a  measure $\tilde{\nu}$ is a free-probability counterpart of $\nu=[a,\sigma^2,M]\in(\mathcal{U}^{<k>},\ast)$, that is $\tilde{\nu}\in (\mathcal{U}^{<k>}
,\boxplus)$,   if and only if
\begin{equation}
V_{\tilde{\nu}}(it)=\frac{a}{2^{k}}+\frac{\sigma^2}{3^{k}}\frac{1}{it}+ \int_{\Rset\setminus{(0)}} x [\Phi(\frac{x}{i t},k,2)  -\frac{1}{1+x^2}\frac{1}{2^{k}}]M(dx).
\end{equation}
Equivalently,
\begin{equation}
V_{\tilde{\nu}}(it)=\frac{a}{2^{k}}+ \int_{\Rset} [\Phi(\frac{x}{i t},k,2)  -\frac{1}{1+x^2}\frac{1}{2^{k}}] \frac{1+x^2}{x} m(dx),
\end{equation}
where $a\in\Rset$, $m(dx):=\frac{x^2}{1+x^2}M(dx)$ on $\Rset\setminus{\{0\}}$ and $ m({0}):=\sigma^2$, is finite Borel measure $m$  and
$ \Phi(z,s,v):=\sum_{n=0}^\infty \frac{z^n}{(v+n)^s}   , \ \ |z|< 1,  v\neq 0, -1,-2,... $
 is the  Hurwitz-Lerch function. Finally, the integrand in (13)  at zero  is equal to  $(3^{k} it)^{-1}$.
\end{prop}
\begin{proof}
Taking into account (11) and Theorem 1, we get the constants  $\textbf{c}=2^{-k}$ and $\textbf{d}=3^{-k}$.
Furthermore, to find $\textbf{g}(z)$ we quote  from Gradshteyn and Ryzhik (1994), formula $(\textbf{9.556})$  that Hurwitz-Lerch function admits the integral  representations:

 if $ \Re v>0, or \, |z|\le 1, z \neq 1,\Re s>0, or \  z=1, \Re s>1$ then
\begin{equation*}
\Phi(z,s,v)=\frac{1}{\Gamma(s)}\int_0^\infty \frac{t^{s-1}e^{-vt}}{1- z e^{-t}}dt.
\end{equation*}
Hence
\begin{multline*}
\textbf{g}(z)=\frac{1}{(k-1)!}\int_0^1 \frac{s (- \log s)^{k-1}}{1+zs} ds \\ = \frac{1}{(k-1)!} \int_0^\infty \frac{w^{k-1} e^{-2w}}{1+zw}dw=\Phi(-z,k, 2),
\end{multline*}
which completes the proof of Proposition 1.
\end{proof}
\begin{rem}
\emph{
(i). \ For the class $(\mathcal{U}^{<1>},\boxplus)$ of  free  s-selfdecomposable measures we may use the  identity
$\Phi(-ix/t,1,2)=it(-x-it\log(1+ix/t))$.
The characterization of the class $ (\mathcal{U}^{<1>} \boxplus)\equiv (\mathcal{U},\boxplus)$ was earlier given in Jurek (2016), Proposition 1(b). Note a misprint there: it should be $t^2$, not $(it)^2$ in the part (b).}

\emph{ (ii). \ Putting $k=0$ in Propositon 1, we get that $(\mathcal{U}^{<0>},\boxplus)\equiv (ID,\boxplus)$, because of the formula (6).
}
\end{rem}
Here are relations between  consecutive classes $\mathcal{U}^{<k>},\boxplus)$:
\begin{cor}
Let us define a differential operator $\mathbb{D}f(t):= 2 f(t) -t \frac{d}{dt}f(t)$. Then for $k\ge 1$
\[
\mathbb{D}: (\mathcal{U}^{<k>)},\boxplus) \to (\mathcal{U}^{<k-1>},\boxplus), \ \mbox{where} \ \
  (\mathcal{U}^{<0>},\boxplus):= (ID,\boxplus).
\]
Hence $\mathbb{D}^{k}: (\mathcal{U}^{<k>},\boxplus)\to (ID,\boxplus).$
\end{cor}
\emph{Proof.}
 Let  $V_{\tilde{\nu}}(it)= \frac{a}{2^{k}}+\frac{1}{3^{k}}\frac{\sigma^2}{it}\in (\mathcal{U}^{<k>},\boxplus)$. Then
\[
 \mathbb{D}(V_{\tilde{\nu}}(it))=\frac{a}{2^{k-1}}+\frac{2\sigma^2}{3^{k}}\frac{1}{it} - t\frac{\sigma^2}{3^{k}}(-1)i(it)^{-2}= \frac{a}{2^{k-1}}+\frac{1}{3^{k-1}}\frac{\sigma^2}{it}\in (\mathcal{U}^{<k-1>},\boxplus).
\]
Since by Wolframalpha.com
$$ \frac{d}{dt}[\Phi(\frac{x}{it},k,2)]=-t^{-1}\big( \Phi(\frac{x}{it}, k-1, 2)- 2\Phi(\frac{x}{it}, k,2)  \big)$$
therefore for  the Poisson part in (12) we have
\begin{multline*}
\mathbb{D}[\Phi(\frac{x}{it},k,2)-\frac{1}{2^{k}}\frac{1}{1+x^2}]=2\Phi(\frac{x}{it},k,2)  -\frac{1}{2^{k-1}}\frac{1}{1+x^2} \\  -t \frac{d}{dt}(\Phi(\frac{x}{it},k,2))=2\Phi(\frac{x}{it},k,2)  -\frac{1}{2^{k-1}}\frac{1}{1+x^2}+\Phi(\frac{x}{it},k-1,2) \\ -2\Phi(\frac{x}{it},k,2) =\Phi(\frac{x}{it},k-1,2)-\frac{1}{2^{k-1}}\frac{1}{1+x^2}, \qquad \qquad \qquad
\end{multline*}
which is the kernel in (12) corresponding to the free-infinitely divisible measure in $(\mathcal{U}^{<k-1>},\boxplus)$.
This completes  a proof of Corollary 1.

\medskip
\textbf{3. Classes $(\mathbb{U}_k, \boxplus)$ of free-infinitely divisible transforms for $ k= 0, 1,2,...$. }

\medskip
For a fixed  $k$,  a probability measure $\mu$ is in  $(\mathbb{U}_k,\ast)$, if there exists a sequence $\nu_n\in (ID,\ast),  n=1,2,...$ such that
\begin{equation}
\rho_n:= T_{\frac{1}{n}} (\nu_1\ast\nu_2\ast ...\ast\nu_n)^{\ast n^{-k}}   \Rightarrow \mu, \ \ \mbox{as} \ \ n\to \infty;
\end{equation}
see  Jurek (1988), Theorem 1.1  and Corollary 1.1. Take there, an operator $Q=I$,  Borel measures $\nu_k$ on the real line and a parameter $\beta=k$.

A class of all possible limits  in (14) is  denoted by $(\mathbb{U}_k,\ast)$  and  measures $\mu\in (\mathbb{U}_k,\ast)$ are referred to as \emph{k-times  s-selfdecomposable measures}. Note that for $k=1$  we get  the class $(\mathbb{U},\ast)$ of s-selfdecomposable measures.

Furthermore,  subclasses $(\mathbb{U}_k,\ast)$  form  an increasing filtration of whole class of  infinitely divisible measures, and all subclasses  admit  random integral representations (see (15) below). Namely,
\begin{multline}
 \mbox{if} \  0  \le k \le l \  \mbox{then} \  (\mathbb{U}_{0},\ast) \subset (\mathbb{U}_k,\ast) \subset (\mathbb{U}_l,\ast) \subset (ID,\ast); \ \mbox{in particular,} \\
  (\mathbb{U}_0,\ast)\equiv (L_0,\ast)\ (\mbox{selfdecomposable measures; see Section 4,below});  \\  \ (\mathbb{U}_1,\ast) \equiv (\mathcal{U}^{<1>},\ast); (\mbox{s-selfdecomposable measures; see Section 2, above}); \\
(\mathbb{U}_k,\ast)= I^{t,t^{k}}_{(0,1]}(ID); \ \ \mbox{and} \ \overline{\cup_{k=1}^\infty(\mathbb{U}_{k},\ast)}= (ID, \ast); \qquad
\end{multline}

\medskip
Here are transforms of  free-infinitely divisible counterparts of measures from  classes $(\mathbb{U}_k,\ast)$:
\begin{prop}
For $k\ge 1$,  a  measure $\tilde{\nu}$ is a free-probability counterpart of $\nu=[a,\sigma^2,M]\in(\mathbb{U}_k,\ast)$, that is $\tilde{\nu}\in (\mathbb{U}_k,\boxplus)$,   if and only if for $t>0$
\begin{multline}
V_{\tilde{\nu}}(it)=\frac{k}{k +1} \,a + \frac{k}{k+2}\frac{\sigma^2}{it}+\int_{\Rset\setminus\{0\}}[ k\, it  \Phi(\frac{x}{it},1,k)- it - \frac{k}{k+1}\frac{x}{1+x^2}]M(dx) \\ =\frac{k}{k +1} \,a  +\int_{\Rset}[ k \, it \big( \Phi(\frac{x}{it},1,k)-k^{-1}\big) - \frac{k}{k+1}\frac{x}{1+x^2}]\frac{1+x^2}{x^2}m(dx)
\end{multline}
where $M$ is arbitrary L\'evy measure and a measure $m$ defined as $m(dx):=\frac{x^2}{1+x^2}M(dx)$ on $\Rset\setminus\{0\}$, and  $m(\{0\}):=\sigma^2$,  is a finite measure and the integrand in (16) at zero is $\frac{k}{k+2}\frac{1}{it}$.

\emph{[$\Phi(z,s,v)$ is the Hurwitz-Lerch function.]}
\end{prop}
\begin{proof}
Since $\mathbb{U}_k=I^{t,t^k}_{(0,1]}(ID)$ therefore we take $a=0,b=1$, $h(t)=t$ and $r(t)=t^k$ in  Theorem 1. Thus $ \textbf{c}=k/(k+1),  \ \textbf{d}=k/(k+2)$ and
$$
\textbf{g}_+(z) = k \int_0^1 \frac{s^k}{1+z s} ds= \frac{k}{k +1}\,{_2F_1}(1,k+1; k+2;- z); \ \mbox{by \textbf{3.194}(5)},
$$
 in Gradshteyn and Ryzhik (1994),($|arg(1+z)|<\pi$) and ${_2F_1}$ denotes \emph{the hypergeometric function}.
It is defined as
$$
{_2F_1}(a,b; c; z): = \sum_{n=0}^\infty\frac{(a)_n(b)_n}{(c)_n}\,\frac{z^n}{n!}, \  \ (x)_n:=x(x+1)...(x+n-1), \ c \neq -\mathbb{N};
$$
where $(x)_n$ is \emph{the  Pochhammer symbol} with the convention $(x)_0:=1$.

Consequently, for the kernel $\textbf{g}_{+}(z)$ we have
\begin{multline*}
 \textbf{g}_+(z) = \frac{k}{k +1}\,{_2F_1}(1,k+1; k+2;- z)\\ =
k (k+1)^{-1}\sum_{n=0}^\infty\frac{(1)_n (k+1)_n}{(k+2)_n} \frac{(-z)^n}{n!}=k \sum_{n=0}^\infty\ \frac{(-z)^n}{k +n+1}= k (-z)^{-1}\sum_{j=1}^\infty \ \frac{(-z)^{j}}{k+j} \\ = k
(-z)^{-1} [\sum_{n=0}^\infty\frac{(-z)^n}{k+n} -\frac{1}{k}]= k (-z)^{-1}[\Phi (-z, 1, k)-k^{-1}];
\end{multline*}
and $\Phi(z, s, a)$ is the Hurwitz-Lerch function.
Finally, we have
$
x\textbf{g}(\frac{ix}{t})= it k ( \Phi(x/(it),1,k)- k^{-1})
$
and this completes a proof of Proposition 2.
\end{proof}

\begin{cor}
If  $\tilde{\nu_k}\in (\mathbb{U}_k,\boxplus)$ then $$\lim_{k \to \infty} V_{\tilde{\nu_k}}(it)= a+ \int_{\Rset}\frac{1+itx}{it-x}m(dx)=V(it) \in (ID,\boxplus),\ \ \mbox{for}\ \  t>0.$$

\noindent In other words, $\overline{\cup_{k=1}^\infty (\mathbb{U}_k, \boxplus)}=(ID,\boxplus)$.
\end{cor}

\noindent To this end,  note that as $k\to \infty$ then

$k\Phi(\frac{x}{it},1,k)= k \sum_{n=0}^\infty (\frac{x}{it})^n \frac{1}{k+n}= \sum_{n=0}^\infty (\frac{x}{it})^n \frac{1}{1+n/k}\to \sum_{n=0}^\infty (\frac{x}{it})^n=\frac{it}{it-x}$,

\noindent and

$  [k it  \Phi(\frac{x}{it},1,k)- it  - \frac{k}{k+1}\frac{x}{1+x^2}]\frac{1+x^2}{x^2}\to [it \frac{it}{it-x}- it -\frac{x}{1+x^2}]\frac{1+x^2}{x^2}=\frac{1+itx}{it-x}$,

\noindent which proves Corollary 2.
\begin{rem}
\emph{For any $\beta\ge -2$, classes $(\mathbb{U}_\beta,\boxplus)$ are well defined by (14); cf. Jurek (1988) and (1989). Here we have restricted indices to the natural numbers to have the sequence of
the inclusions as was announced  in (1). Proposition 2 holds true when one replaces $k\ge 1$ by $\beta>0.$  Furthermore, for $\beta =0$ we get the selfdecomposable distributions as they are discussed below.
}
\end{rem}

\medskip
\textbf{4. Urbanik type classes $(L_k, \boxplus)$ of free-infinitely divisible transforms for $ k= 0, 1,2,...$. }

\medskip
Urbanik (1972 and 1973)  introduced  a filtration of convolution semigroups  of \emph{ selfdecomposable measures} (L\'evy class $L_0$) in a such way that
\begin{multline}
\emph{(Gaussian)}\subset \emph{(stable)}\subset L_\infty \subset ... \subset L_{k+1}\subset L_k \subset ...\subset L_0 \subset ... \subset ID;
\end{multline}
Then using the extreme points method  he  found their descriptions  in terms of  characteristic functions. Measures $\mu  \in (L_k,\ast)$ are called \emph{k-times selfdecomposable}; for a link to Urbanik (1973) see :

$www.math.uni.wroc.pl/^\sim zjjurek/urb-limitLawsOhio1973.pdf$

Later on, all above classes were described in  terms of  random integrals. Namely,
taking
$$r_k(t):= t^{k+1}/(k+1)!, \ t \in (0,\infty), \ \ \mbox{(a time change) and} \  h(t):=e^{-t},$$
we have the following  representations:
\begin{equation*}
L_k=I^{e^{-t},\, r_k(t)}_{(0,\infty)}(ID_{\log^{k+1}}), \ \ ID_{\log^{k+1}}:= \{\nu \in ID : \int_{\Rset}\log^{k+1}(1+|x|)\nu(dx)<\infty\}.
\end{equation*}
Furthermore, from the integral representations one  easily gets their characteristic functions in the same form as in Urbanik (1972 and 1973); see Jurek (1983b), Corollary 2.11 and Theorem 3.1.

\medskip
Here are  the free-infinitely divisible analogues of Urbanik classes $(L_k, \ast)$:
\begin{prop}
For $ k= 0,1,...$, a  measure $\tilde{\nu}$ is a free-probability counterpart of $\nu=[a,\sigma^2, M] \in (L_k,\ast)$, that is $\tilde{\nu}\in (L_k,\boxplus)$,  if and only if
\begin{equation}
V_{\tilde{\nu}}(it)= a+\frac{1}{2^{k+1}}\frac{\sigma^2}{it}+ \int_{\Rset\setminus{(0)}}\big( it Li_{k+1}(\frac{x}{it})-\frac{x}{1+x^2}\big)M(dx), \ \ t>0,
\end{equation}
where a L\'evy measure $M$ has log-moment $\int_{(|x|>1)}\log^{k+1}(1+|x|)M(dx)<\infty$.

\noindent Equivalently,
\begin{equation}
V_{\tilde{\nu}}(it)= a+ \int_{\Rset}\big[ it Li_{k+1}(\frac{x}{it})-\frac{x}{1+x^2}\big]\frac{m(dx)}{\log^{k+1}(1+|x|^{2/(k+1)})}, \ \ t>0,
\end{equation}
where $m$ is a finite Borel measure such that $m(\{0\})= \sigma^2$.
The  integrand in (19) at zero is equal to $\frac{1}{2^{k+1}}\frac{1}{it}$.

\emph{[Here $Li_{s}(z):= \sum_{n=1}^\infty\frac{z^n}{n^s}, \ |z|<1,$ (and analytically continued on $\Cset$) is \emph{the polylogarithmic function}.]}
\end{prop}
\begin{proof}
For $h(s)=e^{-s}, r_k(s)=s^{k+1}/(k+1)!$ and $(a,b]=(0,\infty)$, using Theorem 1, we get $\textbf{c}=1, \textbf{d}=2^{-k-1}$ and
$$
\textbf{g}(z)=\int_0^\infty\frac{e^{-s}}{1+ze^{-s}}\frac{s^k}{k!}ds = -z^{-1}Li_{k+1}(-z), \ \mbox{by Wolframalpha.com}
$$
Hence $\textbf{g}(ix/t)= -t/(ix)Li_{k+1}(-ix/t)=it/x L_{k+1}(x/it)$. Inserting this, together with $\textbf{c}=1$ and  $\textbf{d}=2^{-k-1}$, into (7) in  Theorem 1 we get
\begin{equation}
V_{\tilde{\rho}}(it)= a +\frac{\sigma^2}{it} 2^{-k-1}+\int_{\Rset\setminus{(0)}}[itLi_{k+1}(x/it)-\frac{x}{1+x^2}]M(dx), \ \ t>0.
\end{equation}
Since $m(dx):=\log^{k+1}(1+|x|^{2/(k+1)})M(dx)$ is a finite measure on $\Rset\setminus\{0\}$, (see Jurek and Mason (1993), Proposition 1.8.13.)  and adding an atom $G(\{0\}):=\sigma^2$, we complete the proof of Proposition 3.
\end{proof}
\begin{rem}
\emph{Since $Li_1(\frac{x}{it})\equiv PolyLog[1,\frac{x}{it}]= -\log(1-\frac{x}{it})$ then taking $k=0$ in Proposition 3  above, we retrieve Proposition 2 from Jurek (2016).}
\end{rem}

\medskip
Here is a relation between the consecutive classes $(L_k, \boxplus)$.
\begin{cor}
Let define the differential operator $Df(t):=f(t)-t \frac{d}{dt}f(t)$. Then for $k\ge 0$,
$$D: (L_k,\boxplus)\to ( L_{k-1},\boxplus), \ \ \mbox{ where} \ \  (L_{-1},\boxplus)\equiv (ID,\boxplus).$$
Hence $D^{k+1}: (L_k,\boxplus)\to (ID,\boxplus).$
\end{cor}
\begin{proof} Let  $V_{\tilde{\nu}}(it)= a+\frac{1}{2^{k+1}}\frac{\sigma^2}{it}\in (L_k,\boxplus)$. Then
\[
 D(V_{\tilde{\nu}}(it))=a+\frac{1}{2^{k+1}}\frac{\sigma^2}{it} - t\frac{\sigma^2}{2^{k+1}}(-1)i(it)^{-2}= a+\frac{1}{2^k} \frac{\sigma^2}{it}\in (L_{k-1},\boxplus).
\]
Keeping in mind that $d/dz Li_{k+1}(z)= z^{-1}Li_k(z)$ for the Poissonian part of (18) we have
\begin{multline*}
D[ \int_{\Rset\setminus{(0)}} \big( it Li_{k+1}(\frac{x}{it})-\frac{x}{1+x^2}\big)M(dx)]\\=  \int_{\Rset\setminus{(0)}}[\big( it Li_{k+1}(\frac{x}{it})-\frac{x}{1+x^2}\big)- t \frac{d}{dt} ( it Li_{k+1}(\frac{x}{it}))]M(dx)\\ =  \int_{\Rset\setminus{(0)}}[ - \frac{x}{1+x^2}- t (it \frac{d}{dt}(Li_{k+1}(\frac{x}{it}))]M(dx) \\ =  \int_{\Rset\setminus{(0)}}[ - \frac{x}{1+x^2} - i t^2(\frac{it}{x}Li_k(\frac{x}{it})(- ix) (it)^{-2}]M(dx) \\ =   \int_{\Rset\setminus{(0)}}[ - \frac{x}{1+x^2} + it Li_k(\frac{x}{it})]M(dx) \in (L_{k-1}
, \boxplus),\ \qquad \qquad
\end{multline*}
which completes the proof of Corollary 3.
\end{proof}

\medskip
\textbf{5. Relations between  classes $(L_k,\boxplus)$ and $ (\mathcal{U}^{<k>}, \boxplus)$.}

\medskip
Since $(L_k,\ast)\subset (\mathcal{U}^{<k>},\ast), \mbox{for}\  k=0,1, ...$ (see Jurek (2004), Corollaries 2 and 7),
therefore the injection (4) between the classical and the free infinite divisibility implies that
\begin{equation}
(L_k,\boxplus) \subset (\mathcal{U}^{<k>},\boxplus), \ \  k=0,1,2, ....
\end{equation}

\medskip
Although the classes $(L_k, \ast)$ and $(\mathcal{U}^{<k>},\ast)$ were introduced via the  linear and the non-linear scaling, respectively,  here  is another relation, besides (21), between their free-probability counterparts.

\medskip
 For the notational simplicity, as is in Jurek (1985), let us introduce
\begin{equation*}
\mathcal{I}(\nu) \equiv I^{e^{-t}, t}_{(0,\infty)}(\nu), \quad  \nu\in I D_{\log}; \qquad \mbox{and} \  \ \mathcal{J}(\rho)\equiv I^{t,t}_{(0,1]}(\rho), \ \  \rho \in ID.
\end{equation*}
Then we have
$$(L_0,\ast)=\mathcal{I}(ID_{log}), \ L_{k+1}=\mathcal{I}(I^{e^{-t}, r_k(t)}_{(0,\infty)}(ID_{\log^{k+2}}), \  r_k(t)= \frac{1}{(k+1)!}t^{k+1};$$
$$\mathcal{U}^{<0>}= \mathcal{J}(ID), \ \quad \mathcal{U}^{<k+1>}=\mathcal{J}(I^{t, \tau_k(t)}_{(0,1]}(ID)), \tau_k(t)= \int_0^t (-log x)^{k-1}dx; $$
that is, those classes correspond to the compositions of  $k+1$ mappings $\mathcal{I}$ and $\mathcal{J}$, respectively; see Jurek (2018) for the general theory of compositions of the random integral mappings.

\begin{thm}
For  $ k=0,1,...$, a  measure $\tilde{\rho}\in (\mathcal{U}^{<k>}, \boxplus)$ is in $(L_k, \boxplus)$ if and only if there exists
$\omega \in (ID_{\log^{k+1}},\ast)$ such that $\tilde{\rho}=\widetilde{\mathcal{I}(\omega)}\boxplus \tilde{\omega}$.
\end{thm}
\begin{proof}
 Let $k=0$ and $\tilde{\rho}$ be the free-counterpart of  $\rho\in (\mathcal{U}^{<0>}, \ast) \cap (L_0, \ast)$. Therefore  there exist  $\nu\in ID$ and $\mu \in ID_{\log}$ such that  $\rho=\mathcal{J}(\nu)=\mathcal{I}(\mu)$. However,  to have  such equality it is necessary  and sufficient  that
$\nu = \mu \ast \mathcal{I}(\mu) $; see  Theorem 4.5 in Jurek (1985). Equivalently $\rho= \mathcal{J}(\nu)=  \mathcal{J}(\mu) \ast \mathcal{I}\big(\mathcal{J}(\mu)\big) $.  Taking $\omega:= \mathcal{J}(\mu)$ we have  that $\omega \in ID_{\log} $ as  $\mu\in ID_{\log} $ and finally $\rho=\omega\ast\mathcal{I}(\omega)$. Hence
$
\tilde{\rho}=\widetilde{\big(\mathcal{I}(\omega) \ast \omega \big)} = \widetilde{\mathcal{I}(\omega)}\boxplus \tilde{\omega},
$
which proves the Theorem 2 for $k=0$.

\medskip
Assume that the theorem is true for the classes with  indices $0\le j\le k$ and let  $\tilde{\rho}\in \mathcal{U}^{<k+1>},\boxplus )\cap (L_{k+1},\boxplus)$ be the counterpart of $ \rho\in \mathcal{U}^{<k+1>},\ast )\cap (L_{k+1} ,\ast)$. Then there exist $\nu \in ID$ and $\mu \in ID_{\log^{k+2}}$ such that
$$
\rho=I^{t,\tau_{k+1}(t)}_{(0,1)}(\nu)=\mathcal{J}(I^{t,\tau_{k}(t)}_{(0,1)}(\nu )) \ \mbox{and} \ \ \rho = I^{e^{-t},r_{k+1}(t)}_{(0,\infty)}(\mu)=\mathcal{I}(I^{e^{-t}, r_k(t)}_{(0,\infty)}(\mu)),
$$
and by putting $\nu_1:= I^{t,\tau_{k}(t)}_{(0,1)}(\nu)\in\mathcal{U}^{<k>}$ and $\mu_1:=  \mathcal{I}(I^{e^{-t}, r_k(t)}_{(0,\infty)}(\mu)\in L_k $, from the above line,
we have $\rho=\mathcal{J}(\nu_1)=\mathcal{I}(\mu_1)$. From  this  (as in  the case $k=0$) we get $\nu_1= \mu_1 \ast \mathcal{I}(\mu_1)$ and $\rho=\mathcal{J}(\mu_1)\ast \mathcal{I}(\mathcal{J}(\mu_1))$. Taking $\omega:=\mathcal{J}(\mu_1)\in ID_{\log^{k+1}}$ we get $\tilde{\rho}=\tilde{\omega}\boxplus \widetilde{\mathcal{I}(\omega)}$, which completes a proof.
\end{proof}

Since there is no random integral representation  for the class $(L_\infty,\ast)$, so there is no direct application of the basic Theorem 1. Nevertheless we have
\begin{prop}
 (i) \ A measure $\tilde{\nu}$ is a free-probability counterpart if $\nu\in (L_{\infty}, \ast)$, that is $\tilde{\nu}\in (L_\infty,\boxplus)$, if and only if
\begin{equation}
V_{\tilde{\nu}}(it)= c - \int_{(-2,2]\setminus{\{ 0\}}}\frac{\Gamma(|x|+1) ie^{i \pi x/2} +x}{t^{|x|-1}\,(1-|x|)}\, G(dx),
\end{equation}
 where $c\in \Rset$, $G$ is a finite Borel measure and the integrand at \, $\pm 1$\, is equal to $i\pi/2 \mp \gamma$, respectively.

 (ii)\ A measure $\tilde{\nu}$ is a free-probability counterpart if $\nu\in (\mathcal{U}^{<\infty>}, \ast)$, that is, $\tilde{\nu}\in (\mathcal{U}^{<\infty>},\boxplus)$, if and only if  $V_{\tilde{\nu}}(it)$ is of the form (22) above.
\end{prop}
\begin{proof}
From Urbanik (1972) Theorem 2 or Urbanik (1973), Theorem 2

(or $www.math.uni.wroc.pl/^\sim zjjurek/urb-limitLawsOhio1973.pdf$)

we know that
$\nu\in (L_{\infty},\ast) \ \ \ \mbox{iff}$
\begin{multline*}
\phi_{\nu}(t)= \exp\Big( i a t- \int_{(-2,2]\setminus{\{0\}}} \big[ |t|^{|x|}(\cos (\frac{\pi x}{2})- i \frac{t}{|t|} \sin(\frac{\pi x}{2}))+ itx\big]\frac{G(dx)}{1-|x|}\Big)
\end{multline*}
where $a \in \Rset$ and $G$ is a finite Borel  measure on $(-2,0)\cup (0,2]$.

Let  use the  identification  (4). Then
\begin{multline*}
V_{\tilde{\nu}}(it)=it^2\int_0^\infty\log\phi_{\nu}(-u)e^{-tu}du  = \\  \, it^2\int_0^\infty\Big(- i a u- \int_{(-2,2]\setminus{\{0\}}} \big[ |u|^{|x|}(\cos (\frac{\pi x}{2})+i \frac{u}{|u|} \sin(\frac{\pi x}{2}))-  iux\big]\frac{G(dx)}{1-|x|}\Big)e^{-tu}du\\ =
 a - \int_{(-2,2]\setminus{\{0\}}} it^2 \int_0^\infty\big[ u^{|x|}(\cos (\frac{\pi x}{2})+i \frac{u}{|u|} \sin(\frac{\pi x}{2}))-  iux\big]e^{- tu}du\, \frac{G(dx)}{1-|x|} \\ =
 a - \int_{(-2,2]\setminus{\{0\}}}[ it^2\, \Gamma(1+|x|)t^{-(1+|x|)}(\cos (\frac{\pi x}{2})+i \sin(\frac{\pi x}{2})) +x \big] \frac{G(dx)}{1-|x|}   \\ = a - \int_{(-2,2]\setminus{\{0\}}}[  \Gamma(1+|x|)t^{1-|x|}\, i e^{i \pi x/2} +x \big] \frac{G(dx)}{1-|x|}
\end{multline*}
where
$$
\lim_{x\to 1}\frac{\Gamma(|x|+1) ie^{i \pi x/2} +x}{1-|x|}=i \pi/2  -\gamma; \ \ \lim_{x\to -1}\frac{\Gamma(|x|+1) ie^{i \pi x/2} +x}{1-|x|}=i\pi/2+\gamma,
$$
as for
 $x>0$
\begin{multline*}
\lim_{x\to 1} \frac{d}{dx}
 \Gamma(x+1)= \lim_{x\to1} \frac{d}{dx}\int_0^\infty u^x e^{-u}du=
\lim_{x\to1}  \int_0^\infty \log(u)  u^{x}e^{-u}du\\=\int_0^\infty u \log(u)e^{-u}du=1-\gamma; \ \ \ (\mbox{Euler's constant});
\end{multline*}
which gives part (i)  of Proposition 4. Part (ii) follows from the identity $(L_\infty, \ast)= (\mathcal{U}^{<\infty>}, \ast)$; see Jurek (2004), Corollary 7.
\end{proof}

Because of the special role of $\pm 1$, in Proposition 4, let us  consider the following example:

\begin{ex}\emph{
Let take $G(dx):=1/2\delta_{-1}(dx)+1/2\delta_1(dx)$ (Rademacher distribution) in Proposition 4. Then
$$
V_{\tilde{\nu}}(it)= c - i\pi/2=c + 1/2\int_{\Rset}\frac{1+itx}{it-x}\, \frac{dx}{1+x^2}, \ t>0,
$$
which is the classical example of  Pick function; (Voiculescu representation of a free-infinitely divisible $\tilde{\nu}$); ( $\int_{\Rset}\frac{1+itx}{it-x}\frac{dx}{1+x^2}=- i \pi$).}
\end{ex}

\medskip
\textbf{References.}

\medskip
O. E. Barndorff-Nielsen and  S. Thorbjorsen (2006), Classical and free infinite divisibility and L\'evy processes, \emph{Lect. Notes in Math.} vol. 1866, pp. 33-159; Springer.

\medskip
H. Bercovici and D. V. Voiculescu (1993), Free convolution measures with unbounded support, \emph{Indiana Univ. Math. J.}, vol. 42, pp. 733-773.

\medskip
R. C. Bradely and Z. J. Jurek (2014), The strong mixing and the selfdecomposability, \emph{Stat. $\&$ Probab. Letters}, vol. 84, pp.67-71.

\medskip
W. Feller (1966), \emph{An introduction to probability theory and its applications}, vol. II, J. Wiley $\&$ Sons, New York.

\medskip
B. V. Gnedenko and A. N. Kolomogorov (1954),\emph{ Limit distributions  for sums of independent random variables},
Addison- Wesley.

\medskip
I. S. Gradshteyn and M.  Ryzhik (1994), \emph{Table of integrals, series, and products}, $5^{th}$  Edition, Academic Press, New York.

\medskip
L. Jankowski and Z. J. Jurek (2012), Remarks on restricted Nevanlinna transforms, \emph{Demonstratio Math.} \textbf{XIV}, no 2, pp. 297-307

\medskip
Z. J.  Jurek (1977), Limit distributions for sums of shrunken random variables. In : \emph{Second Vilnius Conf. Probab. Theor. Math. Statistics. Abstract of Communications 3}, pp. 95-96.

\medskip
Z. J. Jurek (1981), Limit distributions for sums of shrunken random variables, \emph{Dissertationes Math.} vol. 185, (46 pages), PWN  Warszawa.

\medskip
Z. J. Jurek (1983a), Limit distributions and one-parameter groups of linear operators on Banach spaces, \emph{J. Multivar. Anal.}, vol. 13, pp. 578-604.

[Also: Mathematisch Instituut, Katholieke Unversiteit, Nijmegen, The Netherlands. Report 8106, February 1981.]

\medskip
Z. J. Jurek (1983b), The classes $L_m(Q)$ of probability measures on Banach spaces,\emph{Bull. Polish Acad. Sci. Math.}, vol. 31, pp. 51-62.

\medskip
Z. J. Jurek (1985), Relations between the s-selfdecomposable and selfdecomposable measures, \emph{ Ann. Probab.} vol. 13, pp. 592-608.

[Also see a Conjucture on : www.math.uni.wroc.pl/$^\sim$zjjurek/Conjecture.pdf]

\medskip
Z. J. Jurek (1988), Random integral representations for classes of limit distributions similar to L\'evy class $L_0$, \emph{Probab. Th. Rel. Fields}, vol. 78, pp. 473-490.

[ Also: Center For Stochastic Process, University of North Carolina, Chapel Hill, North Carolina, Tech. Report 117, September 1985.]

\medskip
Z. J. Jurek (1989),  Random integral representations for classes of limit distributions similar to L\'evy class $L_0$. II.  \emph{Nagoya Math. Journal} \textbf{114}, pp. 53-64.

 [ Also: Center For Stochastic Process, University of North Carolina, Chapel Hill, North Carolina, Tech. Report 139, July 1986.]

\medskip
Z. J. Jurek (2004), Random integral representation hypothesis revisited: new classes of s-selfdecomposable laws;  Proc. International Conf. \emph{ Abstract and Applied  Analysis, Hanoi, Vietnam, 13-17 August 2002};   World Scientific, Singapore, pp. 479-498.

\noindent[Also available on: www.math.uni.wroc.pl/$^\sim$zjjurek/Hanoi2002.pdf ]

\medskip
Z. J. Jurek (2006), Cauchy transforms of measures viewed as some functionals of   Fourier transforms, \emph{Probab. Math. Stat.} vol. 26, no. 1, pp. 187-200.

\medskip
Z. J. Jurek (2007), Random integral representations for free-infinitely divisible and tempered stable distributions, \emph{ Stat.$\&$ Probab. Letters}, vol. 77, pp. 417-425.

\medskip
Z. J. Jurek (2016), On a method of introducing free-infinitely divisible probability measures, \emph{Demonstratio Math.} vol. 49, no 2, pp. 235-251.

\medskip
Z. J. Jurek (2018), Remarks on compositions of some random integral mappings,\emph{Stat. Probab. Letters}  137,  277-282.

\medskip
Z. J. Jurek (2020), On a relation between classical and free infinitely divisible transforms, \emph{Probab. Math. Stat.}, vol.40, no. 2, pp. 439-367.

\medskip
Z. J. Jurek and J.D. Mason (1993), \emph{Operator- limit distributions in probability theory.} Wiley Series in Probability and Mathematical Statistics, New York.

\medskip
Z. J. Jurek and W. Vervaat (1983),. An integral representation for selfdecomposable Banach space valued random variables. \emph{Z. Wahrscheinlichkeitstheorie und verw. Gebiete} \textbf{62} (1983), 347-362.

[ Also: Report 8121, July 1981, Katholieke Universitet, 6525 ED Nijmegen, The Netherlands.]

\medskip
M. Loeve (1963), \emph{Probability theory}, Third Ed., D. Van Nostrand Co., Princeton, New York.

\medskip
K. Urbanik (1972), Slowly varying sequences of random variables, \emph{Bull. de L'Acad.  Polon. Sciences},
Ser. Math.  Astr. Phys., vol. XX, No. 8, pp. 679-682.

\medskip
K. Urbanik (1973), Limit laws for sequences of normed sums satisfying some stability conditions;  \emph{Proc. 3rd International  Symp. on Multivariate Analysis; \emph{Academic  Press}};  Wright State University, Dayton, OH, USA; June 19-24, 1972.

\noindent[Also on:  www.math.uni.wroc.pl/$^\sim$zjjurek/urb-limitLawsOhio1973.pdf ]

\medskip
D.V. Voiculescu (1986), Addition of certain non-commuting random variables, \emph{ J. Funct. Anal.}, vol. 66, pp. 323-346.

\end{document}